\documentclass[12pt]{amsart}

\usepackage{mathrsfs, amssymb,amsthm, amsfonts, amsmath}
\usepackage[all]{xy}
\usepackage{upgreek}
\usepackage{amsmath}

\usepackage{tabularx,multirow,makecell}
\newcolumntype{Y}{>{\centering\arraybackslash}X}
\usepackage{setspace}
\newcounter{cequation}[section]

\newtheorem{theorem}[cequation]{Theorem}
\newtheorem*{theorem*}{Theorem}
\newtheorem{lemma}[cequation]{Lemma}
\newtheorem{corollary}[cequation]{Corollary}

\newtheorem{question}[cequation]{Question}

\theoremstyle{definition}

\newtheorem*{definition*}{Definition}

\theoremstyle{remark}
\newtheorem{remark}[cequation]{Remark}

\makeatletter\@addtoreset{equation}{section}
\makeatother

\makeatother

\usepackage{textcomp}
\usepackage{hyperref}

\newcommand{\CC}{\mathbb{C}}
\newcommand{\RR}{\mathbb{R}}

\newcommand{\ZZ}{\mathbb{Z}}
\newcommand{\PP}{\mathbb{P}}

\newcommand{\WWW}{{\mathscr{W}}}

\newcommand{\ad}{\mathrm{a}}

\newcommand{\Aut}{\operatorname{Aut}}
\newcommand{\GL}{\operatorname{GL}}

\newcommand{\Norm}{\operatorname{Norm}}

\newcommand{\Pic}{\operatorname{Pic}}
\newcommand{\chit}{\chi_{\mathrm{top}}}

\def \ge {\geqslant}
\def \le {\leqslant}

\date{}

\title{Finite groups acting on elliptic surfaces}

\author{Constantin Shramov}

\address{Steklov Mathematical Institute of Russian Academy of Sciences, 8 Gubkina st.,
Moscow, 119991, Russian Federation
\newline
National Research University Higher School of Economics, 
Russian Federation}

\email{costya.shramov@gmail.com}

\thanks{This work is supported by the Russian Science Foundation under grant \textnumero 18-11-00121.}

\begin{document}

\begin{abstract}
We show that automorphism groups of Hopf and Kodaira surfaces have unbounded finite
subgroups. For elliptic fibrations on Hopf, Kodaira, bielliptic, and~$K3$ surfaces, we
make some observations on finite groups acting along the fibers and on the base of such
a fibration.
\end{abstract}



\maketitle

\section{Introduction}

Automorphism groups of geometric objects may have a complicated structure. Sometimes
they appear to be more accessible on the level of their finite subgroups. In particular,
this frequently happens for groups of automorphisms and birational automorphisms
of algebraic varieties, compact complex manifolds, etc.

We say that a group $\Gamma$
\emph{has bounded finite subgroups}
if there exists a constant $B=B(\Gamma)$ such that
for any finite subgroup
$G\subset\Gamma$ one has $|G|\leqslant B$. If this is not the case,
we say that $\Gamma$ has \emph{unbounded finite subgroups}.
There are several interesting situations when automorphism groups and, more generally,
birational automorphism groups of algebraic varieties have bounded finite subgroups.
For instance, this is the case for non-uniruled varieties with vanishing irregularity
over fields of characteristic zero
(see~\cite[Theorem~1.8(i)]{ProkhorovShramov-Bir}),
for many non-ruled projective surfaces over fields of characteristic zero
(see~\cite[Lemma~3.5]{ProkhorovShramov-dim3}),
for varieties over number fields
(see~\cite[Theorem~1.4]{ProkhorovShramov-Bir} together with~\mbox{\cite[Theorem~1.1]{Birkar}}),
and for non-trivial Severi--Brauer
surfaces over fields of characteristic zero containing all roots of unity
(see~\mbox{\cite[Corollary~1.5]{ShramovVologodsky}}).

Our goal is to study finite subgroups of automorphism groups
of \emph{compact complex surfaces}, that is, connected compact complex manifolds of dimension~$2$.
Recall that such a surface is called \emph{minimal} if it does not contain smooth rational
curves with self-intersection equal to~$(-1)$.
For a classification of minimal compact complex surfaces, known as
Enriques--Kodaira classification, we refer the reader to~\mbox{\cite[Chapter~VI]{BHPV-2004}}.

There are several classes of minimal compact complex surfaces that have elliptic fibrations equivariant
with respect to their automorphism groups. These include Kodaira surfaces and other surfaces of algebraic
dimension~$1$; in this case the elliptic fibration is given by algebraic reduction.
Another example is provided by bielliptic surfaces; in this case the elliptic fibration is given by
the Albanese map.

If $\phi\colon X\to C$
is an $\Aut(X)$-equivariant fibration, there appears an exact sequence of groups
\begin{equation}\label{eq:elliptic-fibration}
1\to\Aut(X)_\phi\to\Aut(X)\to\Delta\to 1,
\end{equation}
where the action of the group $\Aut(X)_\phi$ is fiberwise with respect to $\phi$, and
$\Delta$ is a subgroup of $\Aut(C)$.

The goal of this note is to prove the following
result concerning finite groups acting on surfaces of the above types,
and their relation to natural elliptic fibrations on these surfaces.

\begin{theorem}
\label{theorem:main}
\begin{itemize}
\item[(i)]
Let $X$ be either a (primary or secondary) Hopf surface, or a (primary or secondary) Kodaira surface,
or a complex torus, or a bielliptic surface.
Then the automorphism group of $X$ has unbounded finite subgroups.

\item[(ii)]
Let $X$ be a primary Hopf surface of algebraic dimension $1$,
or a two-dimensional complex torus of algebraic dimension $1$, and let $\phi$
be its algebraic reduction.
Then both groups $\Aut(X)_\phi$ and $\Delta$ have unbounded finite subgroups.

\item[(iii)]
Let $X$ be a secondary Hopf surface of algebraic dimension $1$,
and let $\phi$ be its algebraic reduction.
Then the group $\Aut(X)_\phi$ has unbounded finite subgroups.
Furthermore, the group $\Delta$ is finite if and only if $\phi$ has at least three multiple fibers; otherwise $\Delta$ has unbounded finite subgroups.

\item[(iv)]
Let $X$ be a (primary or secondary) Kodaira surface, and let $\phi$
be its algebraic reduction.
Then the group $\Aut(X)_\phi$ has unbounded finite subgroups, while
the group~$\Delta$ is finite.

\item[(v)]
Let $X$ be a complex $K3$ surface of algebraic dimension $1$, and let $\phi$
be its algebraic reduction.
Then the group $\Aut(X)_\phi$ has bounded finite subgroups,
and the group~$\Delta$ is finite.

\item[(vi)]
Let $X$ be a bielliptic surface, and let~$\phi$ be its
Albanese fibration. Then the group~\mbox{$\Aut(X)_\phi$} is finite, while
the group~$\Delta$ has unbounded finite subgroups.
\end{itemize}
\end{theorem}

In all the assertions of Theorem~\ref{theorem:main} one can replace automorphism
groups by birational automorphism groups.
Indeed, the surfaces we consider
are minimal, and are neither rational nor ruled. Thus every birational automorphism
of such a surface is actually biregular, see
for instance~\mbox{\cite[Proposition~3.5]{ProkhorovShramov-CCS}}.

Note that the automorphism group of a complex $K3$ surface has bounded finite subgroups,
see~\cite[Lemma~8.8]{ProkhorovShramov-CCS}. However, this is not enough to
deduce assertion~(v) of Theorem~\ref{theorem:main}. Indeed, a quotient of a group with bounded finite subgroups may
fail to have bounded finite subgroups. To see this one can
take an arbitrary group with unbounded finite subgroups and represent it as a quotient of a free group.
Note also that there exists a $K3$ surface $X$
of algebraic dimension $1$ such that the group $\Aut(X)$ is infinite, see~\cite[Theorem~7.1]{Oguiso}.
Thus, for such a surface the group $\Aut(X)_\phi$ is infinite as well.

All surfaces mentioned in Theorem~\ref{theorem:main} have non-positive Kodaira
dimension. On the other hand, compact complex surfaces of Kodaira
dimension~$1$ are always elliptic. This leads to the
following question.

\begin{question}
What is the analog of Theorem~\ref{theorem:main} for (minimal)
compact complex surfaces of Kodaira dimension~$1$?
\end{question}

There is one class of compact complex surfaces of negative Kodaira dimension whose construction is
quite similar to that of Kodaira surfaces, namely, Inoue surfaces. In particular, their automorphism
groups share certain nice properties, see~\cite{InoueKodaira}.
Although Inoue surfaces are never elliptic (so that most of the observations of the current paper
have no analogs in the case of Inoue surfaces), the following question looks interesting.

\begin{question}
Do automorphism groups of Inoue surfaces have bounded finite subgroups?
\end{question}

\smallskip
\textbf{Acknowledgements.}
I am grateful to S.\,Gorchinskiy, S.\,Nemirovski,  Yu.\,Prokhorov, and E.\,Yasinsky for useful discussions.

\section{Hopf surfaces}
\label{section:Hopf}

In this section we study automorphism groups of Hopf surfaces.
Given a compact complex surface~$X$,
we denote by~$\ad(X)$ the algebraic dimension of~$X$.

Recall that a \emph{Hopf surface} $X$ is a compact complex surface
whose universal cover is
isomorphic to $\WWW=\CC^2\setminus\{0\}$.
Thus $X\cong \WWW/\Gamma$, where $\Gamma\cong \pi_1(X)$
is a group acting freely on $\WWW$.
A Hopf surface $X$ is said to be \emph{primary} if~\mbox{$\pi_1(X)\cong \ZZ$}.
One can show that a primary Hopf surface is
isomorphic to a quotient
$$
X(\alpha,\beta,\lambda,n)=\WWW/\Lambda,
$$
where
$\Lambda\cong\ZZ$ is a group generated by the transformation
\begin{equation}\label{eq:Hopf}
(x,y)\mapsto (\alpha x+\lambda y^n, \beta y)
\end{equation}
for some coordinates $x$ and $y$ on $\CC^2\supset\WWW$.
Here $n$ is a positive integer, and
$\alpha$ and $\beta$ are complex numbers satisfying
$$
0 < |\alpha|\le |\beta|<1;
$$
moreover, one has $\lambda=0$, or $\alpha=\beta^n$.
A \emph{secondary} Hopf surface is a quotient of a primary Hopf surface by a
free action of a finite group. All Hopf surfaces are non-K\"ahler (and in particular non-projective).
If a Hopf surface has algebraic dimension~$1$, then the image of its algebraic reduction is~$\PP^1$.
We refer the reader to~\mbox{\cite[\S10]{Kodaira-structure-2}} for details.

\begin{remark}\label{remark:Hopf-elliptic}
Let $X$ be a Hopf surface. If $\ad(X)=1$, then $X$ is an elliptic surface.
Conversely, if $\ad(X)=0$, then $X$ contains only a finite number of curves (see~\mbox{\cite[Theorem~IV.8.2]{BHPV-2004}}),
and thus $X$ is not elliptic. In particular, if~\mbox{$X=X(\alpha,\beta,\lambda,n)$} is a primary Hopf surface,
then by~\cite[Theorem~30]{Kodaira-structure-2} its algebraic dimension equals $1$ if and only if $\lambda=0$ and $\alpha^k=\beta^r$ for some
positive integers~$k$ and~$r$.
\end{remark}

Automorphism groups of Hopf surfaces are well studied, see~\cite{MatumotoNakagawa} and references therein.
In particular,~\cite[\S2]{MatumotoNakagawa} provides a complete classification of automorphism groups
of secondary Hopf surfaces.
However, we will not use their explicit descriptions.

\begin{lemma}\label{lemma:primary-Hopf}
Let $X$ be a primary Hopf surface. Then the group $\Aut(X)$ has unbounded finite subgroups.
Furthermore, if $X$ has algebraic dimension $1$, and $\phi\colon X\to\PP^1$ is its algebraic reduction, then
in the notation of~\eqref{eq:elliptic-fibration} both groups
$\Aut(X)_\phi$ and $\Delta$ have unbounded finite subgroups.
Moreover, in this case $\Aut(X)_\phi$ contains the group of points of an elliptic curve.
\end{lemma}

\begin{proof}
Let $X=X(\alpha,\beta,\lambda,n)$.

Choose a positive integer $l$, a primitive $l$-th root of unity $\zeta$, and consider the subgroup~\mbox{$\Theta\subset\GL_2(\CC)$} generated by the matrix
\begin{equation}\label{eq:Hopf-unbounded-1}
A=\left(
\begin{array}{cc}
\zeta^n & 0\\
0 & \zeta
\end{array}
\right).
\end{equation}
Then $\Theta\cong\ZZ/l\ZZ$ acts on the universal cover $\WWW$ of $X$, and commutes with the group
$\Lambda\cong\ZZ$ generated by the transformation~\eqref{eq:Hopf}. Thus, we see that $\Theta$ acts faithfully on
$X=\WWW/\Lambda$, so that~\mbox{$\Aut(X)$} has unbounded finite subgroups.

Now suppose that $\ad(X)=1$. Then $\lambda=0$ and $\alpha^k=\beta^r$ for some
positive integers~$k$ and~$r$ by Remark~\ref{remark:Hopf-elliptic}. 
We can choose $k$ and $r$ so that they are coprime.
The algebraic reduction~\mbox{$\phi\colon X\to\PP^1$} is given by the meromorphic function $x^k/y^r$ on $X$.
In other words, the fibers of $\phi$ are images of the subsets
$$
Z_{[\mu:\nu]}=\{(x,y)\mid \mu x^k=\nu y^r\}\subset \WWW, \quad [\mu:\nu]\in\PP^1.
$$
The subsets $Z_{[\mu:\nu]}$ are obviously invariant with respect to
the action of~$\Lambda$.

Consider the action of the multiplicative group $\CC^*$ on
$\WWW$ defined by
\begin{equation}\label{eq:Hopf-unbounded-2}
t\colon (x,y)\mapsto (t^r x,t^k y).
\end{equation}
This action commutes with the group
$\Lambda$, and descends to an action of the group $\CC^*/\bar{\Lambda}$ on~$X$,
where $\bar{\Lambda}$ is a subgroup of $\CC^*$ generated by all
$r$-th roots of $\alpha$ (or, which is the same, by all
$k$-th roots of $\beta$).
Then $\CC^*/\bar{\Lambda}$ is isomorphic to a quotient~\mbox{$\CC^*/\ZZ$} of the group~$\CC^*$
by a subgroup generated by
some $r$-th root of $\alpha$, which in turn is isomorphic to
the group of points of an elliptic curve. It remains to notice that this action
is fiberwise with respect to $\phi$, so that $\Aut(X)_\phi$ contains
the group of points of an elliptic curve, and in particular has unbounded finite subgroups.

Finally, choose a positive integer $l$ and a primitive $l$-th root of unity $\zeta$.
Consider the cyclic subgroup $\Theta_{\PP^1}\subset\GL_2(\CC)$ generated by the matrix
\begin{equation}
A_{\PP^1}=\left(
\begin{array}{cc}
\zeta & 0\\
0 & 1
\end{array}
\right).\
\end{equation}
Then $\Theta_{\PP^1}$ acts on $\WWW$ and commutes with the group
$\Lambda$. Thus, the action of $\Theta_{\PP^1}$ on $\WWW$ descends 
to its action on $X$.
The image $\bar{\Theta}_{\PP^1}$ of $\Theta_{\PP^1}$ in $\Delta$ 
is generated by the transformation
$$
[\mu:\nu]\mapsto [\zeta^{-k}\mu:\nu].
$$
Therefore, $\bar{\Theta}_{\PP^1}$ has order at least $l/k$, and hence 
the group $\Delta$
has unbounded finite subgroups.
\end{proof}

For the following fact about fundamental groups of certain Hopf surfaces
we refer the reader to~\cite[p.~231]{Kato}.

\begin{lemma}\label{lemma:Kato}
Let $X=X(\alpha,\beta,\lambda,n)$ be a primary Hopf surface, and let $G$
be a finite group acting freely on $X$. Let $Y=X/G$, so that
$Y$ is a secondary Hopf surface. Suppose that $\pi_1(Y)\not\subset\GL_2(\CC)$.
Then $n\ge 2$ and $\lambda\neq 0$. Furthermore, there is an isomorphism~\mbox{$\pi_1(Y)\cong\Lambda\times G$}, one has
$G\cong\ZZ/m\ZZ$, and the action of a generator of $G$ on
$\WWW$ is given by the matrix
$$
\left(
\begin{array}{cc}
\xi^n & 0\\
0 & \xi
\end{array}
\right),\
$$
where $\xi$ is some primitive $m$-th root of unity.
\end{lemma}

While elliptic primary Hopf surfaces do not have degenerate fibers, elliptic
secondary Hopf surfaces may have multiple fibers. Those with at most two multiple fibers allow
an explicit description.

\begin{lemma}\label{lemma:Hopf-two-mult-fibers}
Let $Y$ be a secondary Hopf surface of algebraic dimension $1$,
and let~\mbox{$\phi\colon Y\to\PP^1$} be its algebraic reduction.
Suppose that $\phi$ has at most two multiple fibers.
Then~\mbox{$Y\cong\WWW/\Gamma$}, where~$\Gamma$ is generated by
the transformation
\begin{equation*}
(x,y)\mapsto (\alpha x, \beta y),
\end{equation*}
and another transformation of the form
\begin{equation*}
(x,y)\mapsto (\alpha' x, \beta' y)
\end{equation*}
for certain complex numbers $\alpha'$ and $\beta'$.
\end{lemma}

\begin{proof}
See the proof of~\cite[Lemma~8]{Kodaira-structure-2} or the proof of~\cite[Proposition~7.3]{FujimotoNakayama}.
\end{proof}

We will need one elementary observation from linear algebra.

\begin{lemma}\label{lemma:linear-algebra}
Let $\alpha$ and $\beta$ be complex numbers with
$0<|\alpha|,|\beta|<1$, and let
\begin{equation*}
M=\left(
\begin{array}{cc}
\alpha & 0\\
0 & \beta
\end{array}
\right).
\end{equation*}
Then the normalizer $\Norm_{\GL_2(\CC)}(M)$ of the cyclic group
generated by the matrix $M$ in~\mbox{$\GL_2(\CC)$}
coincides with its centralizer.
Furthermore, $\Norm_{\GL_2(\CC)}(M)$ consists of all diagonal matrices in~\mbox{$\GL_2(\CC)$} if $\alpha\neq\beta$, and
$\Norm_{\GL_2(\CC)}(M)=\GL_2(\CC)$ if $\alpha=\beta$.
\end{lemma}

\begin{proof}
The only thing that has to be checked is that in the case $\alpha\neq \beta$ the centralizer of $M$
coincides with $\Norm_{\GL_2(\CC)}(M)$. To see this suppose that there is some matrix~\mbox{$R\in\Norm_{\GL_2(\CC)}(M)$}
that does not commute with $M$. Then~\mbox{$RMR^{-1}=M^{-1}$}, which is impossible since
$$
\det(M)=\alpha\beta\neq\alpha^{-1}\beta^{-1}=\det(M^{-1}).
$$
\end{proof}

Now we prove unboundedness results for secondary Hopf surfaces.

\begin{lemma}\label{lemma:secondary-Hopf}
Let $Y$ be a secondary Hopf surface. Then the group $\Aut(Y)$ has unbounded finite subgroups.
Suppose that $Y$ has algebraic dimension $1$, and $\phi\colon Y\to\PP^1$ is its algebraic reduction.
Then in the notation of~\eqref{eq:elliptic-fibration} the group
$\Aut(Y)_\phi$
contains the group of points of an elliptic curve, and in particular has unbounded finite subgroups.
Furthermore, the group~$\Delta$ is finite if and only if $\phi$ has at least three multiple fibers; otherwise $\Delta$ has unbounded finite subgroups.
\end{lemma}

\begin{proof}
There exists a primary Hopf surface $X=X(\alpha,\beta,\lambda,n)$ and a finite group $G$
acting freely on $X$ such that $Y=X/G$. Note that $\ad(X)=\ad(Y)$, so that $X$ is an elliptic surface
if and only if $Y$ is, cf. Remark~\ref{remark:Hopf-elliptic}. As before, $\WWW$ is the universal cover
of $X$ and $Y$, so that $X=\WWW/\Lambda$, where the generator of the group $\Lambda\cong\ZZ$
acts on $\WWW$ by the transformation~\eqref{eq:Hopf}. In what follows we will use the notation
of the proof of Lemma~\ref{lemma:primary-Hopf}. Choose a positive integer~$l$.

Suppose that $\pi_1(Y)\not\subset\GL_2(\CC)$.
Then the group $\pi_1(Y)$ is described by Lemma~\ref{lemma:Kato}.
We see that  the action of the group $\Theta\cong\ZZ/l\ZZ$ generated by the matrix
$A$ from~\eqref{eq:Hopf-unbounded-1}
commutes with the action of $\pi_1(Y)$ on~$\WWW$. 
Furthermore, if $l$ is coprime to $m$, then
the action of $\Theta$ on~$\WWW$ descends to its faithful action on $Y$. Thus, the group
$\Aut(Y)$ has unbounded finite subgroups in this case.

Suppose that $\pi_1(Y)\subset\GL_2(\CC)$. Choose a primitive $l$-th root of unity $\zeta$,
and consider the subgroup $\Theta\subset\GL_2(\CC)$ generated by the matrix
\begin{equation}\label{eq:scalar}
\left(
\begin{array}{cc}
\zeta & 0\\
0 & \zeta
\end{array}
\right).
\end{equation}
Then $\Theta\cong\ZZ/l\ZZ$ acts on $\WWW$ and commutes with the group
$\pi_1(Y)$. Thus, we see that $\Theta$ acts on
$Y=\WWW/\pi_1(Y)$. This action may be not faithful, but the order of its kernel
is at most~$|G|$. So the group~\mbox{$\Aut(Y)$} has unbounded finite subgroups in this case as well.

Now suppose that $\ad(Y)=1$. Then $\ad(X)=1$, so that $\lambda=0$ and $\alpha^k=\beta^r$ for some
coprime positive integers $k$ and $r$. In particular, we have $\pi_1(Y)\subset\GL_2(\CC)$ by Lemma~\ref{lemma:Kato}.
Consider the matrix
\begin{equation*}
M=\left(
\begin{array}{cc}
\alpha & 0\\
0 & \beta
\end{array}
\right).
\end{equation*}
Note that $\pi_1(Y)$ is contained in the normalizer $\Norm_{\GL_2(\CC)}(M)$. Recall from
Lemma~\ref{lemma:linear-algebra} that the group~\mbox{$\Norm_{\GL_2(\CC)}(M)$} consists of diagonal matrices
if $\alpha\neq \beta$, and coincides with~\mbox{$\GL_2(\CC)$} otherwise. In both cases we see that
the action~\eqref{eq:Hopf-unbounded-2} of the group $\CC^*$ on~$\WWW$
commutes with~$\pi_1(Y)$. As in the proof of Lemma~\ref{lemma:primary-Hopf},
such an action is fiberwise with respect to~$\phi$.
This implies that $\Aut(Y)_\phi$
contains the group of points of some elliptic curve, and thus
has unbounded finite subgroups.

Finally, observe that if $\phi$ has at least three multiple fibers, then the group $\Delta$
is obviously finite. On the other hand, if $\phi$ has at most two multiple fibers, then $Y$
can be constructed as in Lemma~\ref{lemma:Hopf-two-mult-fibers}.
Thus one can easily choose a finite subgroup of the group of scalar matrices acting on $\WWW$
whose image in $\Delta$ has arbitrarily large order.
\end{proof}

\begin{remark}
In the notation of Lemma~\ref{lemma:secondary-Hopf}, suppose that $a(Y)=1$.
Then the scalar matrix~\eqref{eq:scalar} also commutes with $\pi_1(Y)$,
and its image in $\Aut(Y)$ is not contained in the subgroup $\Aut(Y)_\phi$ unless $\alpha=\beta$.
So, if the latter condition does not hold, we conclude that the group $\Delta$ has unbounded
finite subgroups (and in particular $\phi$ has less than three multiple fibers).
\end{remark}

\section{Kodaira surfaces}
\label{section:Kodaira}

In this section we study automorphism groups of Kodaira surfaces.

Recall (see e.g. \cite[\S\,V.5]{BHPV-2004}) that a Kodaira surface
is a compact complex surface of Kodaira dimension $0$ with odd first Betti number.
There are two types of Kodaira surfaces: primary and secondary ones.
A primary Kodaira surface
is a minimal compact complex surface with
the trivial canonical class and first Betti number equal to~$3$, see~\cite[\S6]{Kodaira-structure-1}.
A secondary Kodaira surface is a quotient of a primary Kodaira surface by a
free action of a finite cyclic group, see~\cite[\S\,V.5]{BHPV-2004}.
The algebraic dimension of primary and secondary Kodaira surfaces
is always equal to~$1$. The image of the algebraic reduction of a primary Kodaira surface is an elliptic curve,
while for a secondary Kodaira surface
it is~$\PP^1$, see~\cite[\S\,VI.1]{BHPV-2004}.

Let $X$ be a primary Kodaira surface.
The universal cover of $X$ is isomorphic to $\CC^2$, and the fundamental group $\Gamma=\pi_1(X)$
has the following presentation:
\begin{equation*}
\Gamma=\langle \updelta_1,\updelta_2,\updelta_3,\upgamma
\mid \updelta_1\updelta_2\updelta_1^{-1}\updelta_2^{-1}= \updelta_3^r,\
\updelta_i\updelta_3\updelta_i^{-1}\updelta_3^{-1}=\updelta_i\upgamma\updelta_i^{-1}\upgamma^{-1}=1 \text{\ for\ } i=1,2,3
\rangle,
\end{equation*}
where $r$ is a positive integer, see~\cite[\S6]{Kodaira-structure-1}.

According to~\cite[pp.~787--788]{Kodaira-structure-1}, the surface~$X$ can be obtained as a quotient
of $\CC^2$ by the group $\Gamma$ whose action is defined by
\begin{equation}\label{eq:Kodaira-quotient}
\begin{aligned}
& \updelta_1\colon (z,w)\mapsto (z+a_1, w+\bar{a}_1z+b_1),\\
& \updelta_2\colon (z,w)\mapsto (z+a_2, w+\bar{a}_2z+b_2),\\
& \updelta_3\colon (z,w)\mapsto \left(z, w+b_3\right),\\
& \upgamma\colon (z,w)\mapsto (z,w+b_4),\\
\end{aligned}
\end{equation}
where $\bar{a}_1a_2-\bar{a}_2a_1=rb_3\neq 0$,
and the complex numbers $b_3$ and $b_4$ are
linearly independent over~$\RR$.
The algebraic reduction $\phi\colon X\to C$ corresponds to the projection
of $\CC^2$ on the $z$-coordinate.

We start with a straightforward observation concerning automorphisms of primary Kodaira surfaces.

\begin{lemma}[{cf.~\cite[Remark~2]{Borcea}}]
\label{lemma:primary-Kodaira}
Let $X$ be a primary Kodaira surface, and let~\mbox{$\phi\colon X\to C$} be its algebraic reduction. Then
in the notation of~\eqref{eq:elliptic-fibration} the group
$\Aut(X)_\phi$ contains the group of points of an elliptic curve. In particular,
$\Aut(X)_\phi$, and thus also~$\Aut(X)$,
has unbounded finite subgroups.
\end{lemma}

\begin{proof}
Let $\Gamma=\pi_1(X)$.
The action of $\Gamma$ on the universal cover $\CC^2$ of $X$ is given
by formulas~\eqref{eq:Kodaira-quotient}.

Consider the action of the additive group $\CC$ on $\CC^2$ given by
\begin{equation}\label{eq:additive}
t\colon (z,w)\mapsto (z,w+t).
\end{equation}
This action commutes with $\Gamma$, and thus descends to
a faithful action of the group of points of the elliptic curve $\CC/\Lambda$ on
$X$, where $\Lambda$ is a lattice generated by $b_3$ and $b_4$.
It remains to notice that the latter action is fiberwise with
respect to~$\phi$.
\end{proof}

Any automorphism
$g$ of a primary Kodaira surface $X$ gives an automorphism
$\bar{g}$ of the base of the algebraic reduction~\mbox{$\phi\colon X\to C$}.
Since $C$ is an elliptic curve, one has~\mbox{$H^0(C, T_C)\cong\CC$}, and
$\bar{g}_*$ acts on $H^0(C, T_C)$ as a multiplication by some
complex number~$\alpha(g)$ with~\mbox{$|\alpha(g)|=1$}.
The following is a particular case of \cite[Proposition~6.4]{FujimotoNakayama}
(see also \cite[Lemma~6.1]{FujimotoNakayama}).

\begin{lemma}\label{lemma:Aut-Kodaira-explicit}
Let $X$ be a primary Kodaira surface, and let
$C$ be the base of its algebraic reduction. Let $g$ be an automorphism of
$X$ such that the corresponding automorphism
of~\mbox{$H^0(C, T_C)$} is the multiplication
by $\alpha\in\CC^*$. Then $g$ is induced
from the automorphism of the universal cover $\CC^2$ of $X$
of the form
\begin{equation}\label{eq:Kodaira-explicit}
(z,w)\mapsto (\alpha z+\beta, w+F(z)),
\end{equation}
where $\beta$ is some complex number such that $\beta=0$
if $\alpha=1$, and $F(z)$ is some holomorphic
function.
\end{lemma}

\begin{corollary}\label{corollary:primary-Kodaira-translation-on-C}
Let $X$ be a primary Kodaira surface, and let~\mbox{$\phi\colon X\to C$} be its algebraic reduction.
Let $g$ be an automorphism of $X$ that acts by a translation on
$C$. Then $g$ is fiberwise with respect to $\phi$ (so that the action of $g$ on $C$ is trivial).
\end{corollary}

\begin{proof}
In the notation of Lemma~\ref{lemma:Aut-Kodaira-explicit}, one has
$\alpha=1$. Therefore, we see from~\eqref{eq:Kodaira-explicit}
that the action of $g$ is fiberwise with respect to~$\phi$.
\end{proof}

\begin{corollary}\label{corollary:primary-Kodaira-Delta}
Let $X$ be a primary Kodaira surface, and let~\mbox{$\phi\colon X\to C$} be its algebraic reduction.
Then in the notation of~\eqref{eq:elliptic-fibration} the group $\Delta$ is finite.
\end{corollary}

\begin{proof}
It follows from Corollary~\ref{corollary:primary-Kodaira-translation-on-C} that
$\Delta$ does not contain elements that act by translations on~$C$.
\end{proof}

Now we deal with secondary Kodaira surfaces.

\begin{lemma}\label{lemma:secondary-Kodaira}
Let $Y$ be a secondary Kodaira surface, and let~\mbox{$\phi\colon Y\to \PP^1$} be its algebraic reduction.
Then
in the notation of~\eqref{eq:elliptic-fibration} the group
$\Aut(Y)_\phi$ contains the group of points of an elliptic curve
(so that in particular $\Aut(Y)_\phi$ and $\Aut(Y)$ have unbounded finite subgroups).
On the other hand, the group $\Delta$ is finite.
\end{lemma}

\begin{proof}
There exists a primary Kodaira surface $X$
such that $Y$ is a quotient of $X$ by a finite
cyclic group $G$.
By Lemma~\ref{lemma:Aut-Kodaira-explicit},
the action of a generator of $G$ is induced by an automorphism
of the universal cover $\CC^2$ of $X$ (and $Y$)
of the form~\eqref{eq:Kodaira-explicit}. The latter transformation
commutes
with the action of the additive group $\CC$ on $\CC^2$
given by~\eqref{eq:additive}. Similarly to the proof of
Lemma~\ref{lemma:primary-Kodaira}, this gives an action of
the group of points of an elliptic curve on $Y$ that is fiberwise
with respect to~$\phi$.

To prove that
the group $\Delta$ is finite it is enough to notice that
$\phi$ has at least~$3$ multiple fibers (see for instance
the proof of~\mbox{\cite[Lemma~4.4]{FujimotoNakayama}}, or
the proof of~\mbox{\cite[Lemma~7.1]{ProkhorovShramov-CCS}}).
\end{proof}

\section{$K3$ surfaces, complex tori, and bielliptic surfaces}
\label{section:K3}

In this section we consider automorphism groups of complex $K3$ surfaces, complex tori, and bielliptic surfaces.
Recall that the image of the algebraic reduction of a complex $K3$ surface (as well as the image of
any other elliptic fibration on a $K3$ surface) is~$\PP^1$.

The following assertion is well-known to experts.
We provide its proof for the reader's convenience.

\begin{lemma}\label{lemma:K3-isotrivial}
Let $X$ be a complex $K3$ surface,
and let $\phi\colon X\to \PP^1$ be an elliptic fibration.
Then $\phi$ has at least three singular fibers.
\end{lemma}

\begin{proof}
Let $F_1,\ldots,F_k$ be all singular fibers of $\phi$.
Denote by $n_i$ the number of irreducible components of~$F_i$.
One can see from the classification of
degenerate fibers of elliptic fibrations (see~\cite[\S\,V.7]{BHPV-2004})
that $\chit(F_i)\le n_i+1$. This gives
\begin{equation}\label{eq:24}
24=\chit(X)=\sum \chit(F_i)\le \sum (n_i+1)=\sum n_i +k.
\end{equation}

Now for every $i$ choose $C_i$ to be one of the irreducible components
of $F_i$, and consider the lattice $\Lambda$ in $\Pic(X)$ generated by all
the irreducible components of all singular fibers~$F_i$ except for the curves $C_i$, $1\le i\le k$.
Then the rank of $\Lambda$ equals $\sum (n_i-1)$, and the intersection form
is negative-definite on $\Lambda$. On the other hand, the signature of the intersection
form on the lattice $H^2(X,\ZZ)$ is $(3,19)$, see e.g.~\mbox{\cite[Chapter~VIII]{BHPV-2004}}.
This gives
$$
19\ge \sum (n_i-1)=\sum n_i -k.
$$
Combining the latter inequality with~\eqref{eq:24}, we obtain
$2k\ge 5$, so that $k\ge 3$.
\end{proof}

In the case when $X$ is a $K3$ surface of algebraic dimension $1$, the assertion of 
Lemma~\ref{lemma:K3-isotrivial} is implied by a much more general~\mbox{\cite[Proposition~A.1]{CHL}}.

\begin{corollary}\label{corollary:K3}
Let $X$ be a complex $K3$ surface,
and let $\phi\colon X\to \PP^1$ be an elliptic fibration.
Let~\mbox{$\Gamma\subset\Aut(X)$} be a group such that $\phi$ is $\Gamma$-equivariant.
Let $\Gamma_\phi\subset\Gamma$ be the subgroup of all transformations
that are fiberwise with respect to $\phi$, and 
let~\mbox{$\Delta=\Gamma/\Gamma_\phi$}.
Then the group~$\Gamma_\phi$ has bounded finite subgroups, while 
the group~$\Delta$ is finite.
\end{corollary}

\begin{proof}
The assertion for the group $\Gamma_\phi$ follows from the fact that the whole group
$\Aut(X)$ has bounded finite subgroups, see e.g.~\cite[Lemma~8.8]{ProkhorovShramov-CCS}.
Furthermore, by Lemma~\ref{lemma:K3-isotrivial} the fibration $\phi$ has at least three
singular fibers, so that there is a finite $\Delta$-invariant subset of~$\PP^1$ of cardinality at least three.
The latter is impossible for an infinite group~$\Delta$.
\end{proof}

\begin{remark}\label{remark:tori}
Let $X$ be a non-projective complex torus of positive algebraic dimension, and let
$\phi\colon X\to C$ be its algebraic reduction. Then $C$ is a complex torus as well.
Fix a point $O\in C$, and a point $\hat{O}\in X$ such that
$\phi(\hat{O})=O$. Then $\phi$ can be thought of as the quotient morphism of $X$ by some complex subtorus $E\subset X$.
Let us use the notation of~\eqref{eq:elliptic-fibration}.
Since the group of points of $E$ is contained in $\Aut(X)_\phi$, we see that
$\Aut(X)_\phi$ has unbounded finite subgroups. Furthermore, since
$\phi$ induces a surjective homomorphism on the groups of torsion points of $X$ and $C$,
we see that the group $\Delta$ has unbounded finite subgroups as well.
\end{remark}

We will need the following well known fact concerning elliptic curves.

\begin{lemma}\label{lemma:elliptic-curve}
Let $C$ be an elliptic curve, and $G\subset\Aut(C)$ be a non-trivial finite group. Suppose that the action of
$G$ on $C$ is not free. Then the normalizer of $G$ in $\Aut(C)$ is finite.
\end{lemma}

\begin{proof}
This follows from the existence of the exact sequence
$$
1\to C\to \Aut(C)\to \ZZ/n\ZZ\to 1,
$$
where $C$ is identified with the group of its points acting by translations on $C$,
and~\mbox{$n\le 6$}.
\end{proof}

Now we pass to the case of bielliptic surfaces. Note that such a surface is always projective.
However, its Albanese map provides a natural structure of an elliptic fibration over an elliptic curve.

\begin{lemma}\label{lemma:bielliptic}
Let $X$ be a bielliptic surface,
and $\phi\colon X\to C$ be its Albanese fibration. Then the group
$\Aut(X)$ has unbounded finite subgroups. Furthermore,
in the notation of~\eqref{eq:elliptic-fibration} the group
$\Aut(X)_\phi$ is finite, while
the group $\Delta$ contains the group of points of an elliptic curve (and in particular has unbounded finite subgroups).
\end{lemma}

\begin{proof}
One can construct $X$ as a quotient $E\times F/G$, where $E$ and $F$
are elliptic curves, and~$G$ is a finite group acting faithfully on $E$ by translations
and acting faithfully on $F$ \emph{not only} by translations, see for instance~\cite[\S\,V.5]{BHPV-2004}.

Since the group of points of $E$ acting on $E\times F$ commutes with $G$, there
is a homomorphism (with a finite kernel) from this group
to $\Aut(X)$. In particular, the group
$\Aut(X)$ has unbounded finite subgroups.
Furthermore, one has $C\cong E/G$, and the following diagram is commutative and $E$-equivariant:
$$
\xymatrix{
E\times F\ar@{->}[r]\ar@{->}[d]  & X\ar@{->}[d]^{\phi}\\
E\ar@{->}[r] & C
}
$$
In particular, the group $\Delta$ contains the group of points of the elliptic curve $C$, and thus has unbounded finite subgroups.

Consider the subgroup $\widetilde{\Aut}(X)_\phi\subset\Aut(E\times F)$
defined by the exact sequence
$$
1\to G\to\widetilde{\Aut}(X)_\phi\to\Aut(X)_\phi\to 1.
$$
By construction, the group
$G$ is a normal subgroup in $\widetilde{\Aut}(X)_\phi$. Thus it follows
from~\mbox{\cite[Lemma~2.1]{BennettMiranda}} that
$$
\widetilde{\Aut}(X)_\phi\subset\Aut(E)\times\Aut(F)\subset\Aut(E\times F),
$$
and hence the projection $E\times F\to F$ is $\widetilde{\Aut}(X)_\phi$-equivariant.
The image $\overline{\Aut}(X)_\phi$ of $\widetilde{\Aut}(X)_\phi$ in $\Aut(F)$ normalizes the group $G\subset\Aut(F)$,
hence $\overline{\Aut}(X)_\phi$ is finite by Lemma~\ref{lemma:elliptic-curve}.
On the other hand, the kernel $R$ of the homomorphism
$$
\widetilde{\Aut}(X)_\phi\to\overline{\Aut}(X)_\phi
$$
has order at most $|G|$, because the image of $R$ in $\Aut(X)_\phi$ is trivial. This
means that the group $\widetilde{\Aut}(X)_\phi$ and thus also
its quotient~\mbox{$\Aut(X)_\phi$} is finite.
\end{proof}

An alternative proof of the fact that in the notation of Lemma~\ref{lemma:bielliptic} the group~$\Aut(X)_\phi$ is finite
can be obtained from the explicit description
of the group $\Aut(X)$, which appears to contain the group of points of $E$ as a subgroup of finite index, see~\mbox{\cite[Table~3.2]{BennettMiranda}}.

Now we complete the proof of Theorem~\ref{theorem:main}.
Assertion~(ii) follows from Lemma~\ref{lemma:primary-Hopf}
together with Remark~\ref{remark:tori}.
Assertion~(iii) is given by Lemma~\ref{lemma:secondary-Hopf}.
Assertion~(iv) follows from Lemma~\ref{lemma:primary-Kodaira},  Corollary~\ref{corollary:primary-Kodaira-Delta},  and Lemma~\ref{lemma:secondary-Kodaira}.
Assertion~(v) is given by Corollary~\ref{corollary:K3}.
Assertion~(vi) is given by Lemma~\ref{lemma:bielliptic}.
Finally, assertion~(i) follows from assertions~(ii), (iii), and~(iv) together with Lemma~\ref{lemma:bielliptic}.


\end{document}